\pdfoutput=1
\documentclass[12pt,reqno]{amsart}
\usepackage{amsthm,amssymb,bbm,color,enumerate}
\usepackage{amsfonts}

\usepackage{latexsym}
\usepackage{amsmath}
\usepackage{amsthm}
\usepackage{amssymb}
\usepackage{graphicx}
\usepackage{amsfonts}
\usepackage{comment}
\usepackage{hyperref}

\theoremstyle{plain}
\newtheorem{theorem}{Theorem}
\newtheorem{thm}[theorem]{Theorem}
\newtheorem{lem}[theorem]{Lemma}
\newtheorem{prop}[theorem]{Proposition}

\theoremstyle{definition}

\numberwithin{theorem}{section}
\numberwithin{equation}{section}

\newcommand*{\doo}{\partial}

\newcommand{\norm}[1]{\left\| #1 \right\|}

\DeclareMathOperator{\dist}{dist}

\DeclareMathOperator{\real}{Re}

\numberwithin{equation}{section}

\newcommand{\wt}{\widetilde}
\newcommand{\cj}{\overline}
\newcommand{\pa}{\partial}
\newcommand{\fr}{\frac}

\newcommand{\RR}{\mathbb R}
\newcommand{\NN}{\mathcal N}

\newcommand{\al}{\alpha}
\newcommand{\del}{\delta}

\allowdisplaybreaks[1]

\def\vint_#1{\mathchoice%
          {\mathop{\kern 0.2em\vrule width 0.6em height 0.69678ex
depth -0.58065ex
                  \kern -0.8em \intop}\nolimits_{\kern -0.4em#1}}%
          {\mathop{\kern 0.1em\vrule width 0.5em height 0.69678ex
depth -0.60387ex
                  \kern -0.6em \intop}\nolimits_{#1}}%
          {\mathop{\kern 0.1em\vrule width 0.5em height 0.69678ex
              depth -0.60387ex
                  \kern -0.6em \intop}\nolimits_{#1}}%
          {\mathop{\kern 0.1em\vrule width 0.5em height 0.69678ex
depth -0.60387ex
                  \kern -0.6em \intop}\nolimits_{#1}}}
\def\vintslides_#1{\mathchoice%
          {\mathop{\kern 0.1em\vrule width 0.5em height 0.697ex depth -0.581ex
                  \kern -0.6em \intop}\nolimits_{\kern -0.4em#1}}%
          {\mathop{\kern 0.1em\vrule width 0.3em height 0.697ex depth -0.604ex
                  \kern -0.4em \intop}\nolimits_{#1}}%
          {\mathop{\kern 0.1em\vrule width 0.3em height 0.697ex depth -0.604ex
                  \kern -0.4em \intop}\nolimits_{#1}}%
          {\mathop{\kern 0.1em\vrule width 0.3em height 0.697ex depth -0.604ex
                  \kern -0.4em \intop}\nolimits_{#1}}}

\newcommand{\aveint}[2]{\mathchoice%
          {\mathop{\kern 0.2em\vrule width 0.6em height 0.69678ex
depth -0.58065ex
                  \kern -0.8em \intop}\nolimits_{\kern -0.45em#1}^{#2}}%
          {\mathop{\kern 0.1em\vrule width 0.5em height 0.69678ex
depth -0.60387ex
                  \kern -0.6em \intop}\nolimits_{#1}^{#2}}%
          {\mathop{\kern 0.1em\vrule width 0.5em height 0.69678ex
depth -0.60387ex
                  \kern -0.6em \intop}\nolimits_{#1}^{#2}}%
          {\mathop{\kern 0.1em\vrule width 0.5em height 0.69678ex
depth -0.60387ex
                  \kern -0.6em \intop}\nolimits_{#1}^{#2}}}

\title[Reconstructing obstacles using CGO solutions]{Reconstructing unknown inclusions for the biharmonic equation}


\author{Gyeongha Hwang, Manas Kar}
\address{Department of Mathematics, Yeungnam University, 280 Daehak-Ro, Gyeongsan, Gyeongbuk 38541, Republic of Korea}
\email{ghhwang@yu.ac.kr}
\address{Department of Mathematics, Indian Institute of Science Education and Research (IISER) Bhopal, Bhopal-462066, India}
\email{manas@iiserb.ac.in}

\begin{document}

\maketitle

\begin{abstract}
Herein, we study an inverse problem for detecting  unknown obstacles by the enclosure method using the Dirichlet--to--Neumann map for measurements. We justify the method for an penetrable obstacle case involving a biharmonic equation. We use complex geometrical optics solutions with a logarithmic phase to reconstruct some non--convex parts of the obstacle. 

 \end{abstract}
\section{Introduction}
The inverse problem in this article involves determining an unknown obstacle or jump in the inclusions embedded in a known background medium via near--field measurements. There are several methods proposed to understand the reconstruction issue for inverse obstacle problems. We can mention  the linear sampling method of Colton-Kirsch \cite{Colton:Kirsch:1996}, probe method of Ikehata \cite{MR1642619}, factorization method of Kirsch \cite{MR1662460}, enclosure method of Ikehata \cite{MR1830379}, and singular sources method of Potthast \cite{MR1737077}. In these classical analytic methods, the probe and singular sources methods use an approximation of singular solutions to probe a region of interest; whereas the linear sampling and factorization methods use the far-field pattern of the fundamental solution as a given part of the far-field equation. 
The enclosure method employs the complex geometrical optics solutions (CGO) which were used in the Calder\'on method and the uniqueness proof by Sylvester-Uhlmann \cite{MR0873380}, full coefficient reconstruction by Nachman \cite{MR0970610}. In contrast to the probe and singular sources method, the CGO solutions are given by a constructive method. The enclosure method can be classified in two types. In the first type, one needs to have an infinite measurement and the second type uses a single or at most two measurements. An extensive literature of this topic can found in the survey article \cite{MR1830379}. In this work, we use CGO solutions to discuss the enclosure method for the zeroth--order perturbation of the biharmonic--type operator.
  
Let $\Omega \subset \RR^3$ be a bounded, smooth domain. We assume $D (\subset\subset \Omega)$ to be an unknown obstacle with a $C^2$--regular boundary $\partial D$, such that $\Omega\setminus \overline{D}$ is connected. As a model problem, we consider the zeroth--order perturbation of the biharmonic equation with the Navier boundary condition: 
\begin{align}\label{main_eqn}
\left\{
\begin{array}{ll}
\Delta^2 u + \wt n u = 0 &\text{in } \Omega\\
u = f_1 & \text{on } \partial \Omega\\
\Delta u = f_2 & \text{on } \partial \Omega.
\end{array}
\right.
\end{align}
We assume that $\wt n(x) = 1 + n_D(x)\chi_D(x)$, for all $x \in \Omega$, such that $\wt n\in L^{\infty}(\Omega)$. Here, $\chi_D$ is the characteristic function of $D$. Let us also assume that $n_D \in L^\infty_+(D)$, where
$
L^\infty_+(D) := \{f\in L^{\infty}(D) ; f\geq C>0 \ \text{for some $C\in \mathbb{R}$} \}.
$ 

%In practice, the biharmonic equation appears in many areas of physics, such as in the elasticity theory, plate plasma, and Stokes flow \cite{Gazzola}.
%In \cite{kang:2012}, they posed inverse boundary value problems for thin elastic plates in the planner domain for the system
%$
%\div\div(\mathbb{C}\nabla^2u) = 0,
%$
%where $\mathbb{C}$ is the fourth-order tensor. Recently, an anisotropic nonlocal fractional $p$-biharmonic system was considered \cite{KRZ}. They studied the existence and uniqueness of weak solutions to the associated interior source and exterior value problems, unique continuation properties, monotonicity relations, and inverse problems for exterior Dirichlet--to--Neumann maps. 
% We denote $\LL$ as the operator
% \[
% \LL u :=  \Delta^2 u  +  \wt n u.
% \]
%  The domain of definition of the operator $\LL$ is 
% $$\DD(\LL) := \{u \in H^4(\Omega); u|_{\pa \Omega} = (\Delta u)|_{\pa \Omega} = 0\},$$
% which is a closed operator on $L^2(\Omega)$, see \cite{Grubb}, see also the introduction of \cite{KRUPCHYK20121781}. We assume that $0$ is not an eigenvalue of the operator 
% $
% \LL : \DD(\LL) \to L^2(\Omega),
% $
% related to the problem \eqref{main_eqn}. 

By the standard well--posedness of the boundary value problem for the fourth--order elliptic equation, problem \eqref{main_eqn} has a unique solution $u\in H^4(\Omega)$ for any $f_1 \in H^{7/2}(\doo\Omega)$ and $f_2 \in H^{3/2}(\doo\Omega)$; see Lemma \ref{regularityAPP}. We could refer  \cite{Gazzola} for some applications of this type of model. We define the Dirichlet--to--Neumann map corresponding to the biharmonic problem above as follows:
\[
\NN_{D} : H^{7/2}(\pa\Omega) \times H^{3/2}(\pa\Omega) \to H^{5/2}(\pa\Omega) \times H^{1/2}(\pa \Omega)
\]
by 
\[
\NN_{D} \left(f_1, f_2\right) = \left(\fr{\pa u}{\pa \nu}|_{\pa \Omega}, \fr{\pa}{\pa \nu}(\Delta u)|_{\pa \Omega}\right),
\]
where $u$ is the solution to \eqref{main_eqn}, and $\nu$ denotes the outward unit normal vector to $\pa \Omega$.
% Since, the forward problem \eqref{main_eqn} is well-posed in $H^4(\Omega)$, so the Dirichlet-to-Neumann map $\NN_{\wt \ga, \wt A, \wt n}$ is well-defined on the appropriate fractional Sobolev spaces $H^{7/2}(\pa\Omega) \times H^{3/2}(\pa\Omega)$. 
The inverse problem in this study is to determine the shape and location of $D$ from the knowledge of the Dirichlet--to--Neumann map $\NN_{D}$ measured at boundary $\partial\Omega$.

We mainly applied the enclosure method proposed by Ikehata \cite{Ikehata:2000} to solve the inverse problem. See also \cite{Ikehata:1999} for the enclosure method using a single input. In \cite{Ikehata:2000}, the author considered the inverse problem for a conductivity equation as a model problem and used complex geometrical optics solutions with a linear phase to detect the convex hull of an obstacle. Many studies have been conducted on detecting unknown obstacles using this enclosure method for various other types of partial differential equations. See, for example the Maxwell system \cite{Kar:Sini:2014, MR3396397, Zhou:2010}, elasticity equation \cite{Kar:SIni:elas, MR2853548}, Helmholtz equation \cite{Sini:Yoshida:2012, MR3195891, MR2367868, MR2765688, MR2390978}, \cite{MR3040565}, and a review paper \cite{MR3098659}. 
We could also mention the works \cite{MR4245819, MR4281493} in case of complex conductivity and \cite{MR3562270} to deal with the anisotropic conductivity.
Meanwhile, the authors \cite{MR3773773, MR3810150, MR3320025} studied inverse problems for detecting unknown inclusions for $p$-Laplace type model by utilising enclosure method.
In \cite{Sini:Yoshida:2012}, the authors considered the problem of determining the unknown obstacle for divergence form elliptic equations with lower--order terms from the knowledge of the Dirichlet--to--Neumann map. The result was proved using Meyers $L^p$--estimates \cite{Meyers1963} to eliminate some geometrical assumptions on the obstacle surface. The ideas of using $L^p$--estimates to prove the enclosure method for the reconstruction, have also been implemented in \cite{Kar:SIni:elas} and \cite{Kar:Sini:2014}. 

The Calder\'on problem, in particular, to prove its uniqueness issue corresponding to the biharmonic operator was initiated by Isakov for general $L^{\infty}$ potential \cite{MR1120907}, Ikehata for general $L^p$ potential \cite{MR1127214}. 
%The methods of the construction of the required CGO solutions with a linear phase are different from each other. It should be noted about the construction method of  in Krupchyk-Lassas-Uhlmann \cite{KRUPCHYK20121781}.
The inverse
problem corresponding to the biharmonic operator was  studied by Krupchyk--Lassas--Uhlmann \cite{KRUPCHYK20121781} to prove the unique determination of the first--order perturbation $-i^{-1}A(x)\cdot \nabla + q$ of the biharmonic operator from the Dirichlet--to--Neumann map measured on a part of the boundary using the CGO solution with nonlinear phase. Further results on the inverse problems for the biharmonic equation can be found in various studies; see, for example, \cite{YK}, \cite{MR3118392}, \cite{MR3484382}, and the references therein.

The enclosure method is based on the asymptotic behavior of $I_{x_0}(h,t)$. The main idea behind this method is as follows. First, we define an indicator function, $I_{x_0}(h, t)$, as in \eqref{indicatorFUN}. The indicator function represents the energy difference between when an obstacle is in $\Omega$ and when no obstacle is in $\Omega$. Subsequently, an asymptotic estimate of the indicator function for a small parameter, $h>0$, is studied (see Theorem \ref{main thm}). This indicates whether the level set of the phase function touches the obstacle surface. Finally, the intersection of all level sets touching the interface determines convex hull of the obstacle and its non--convex part.

Most of our efforts are devoted to the proof of Theorem \ref{main thm}. Because of Lemma \ref{Pf::1}, providing the lower and upper bounds of the indicator function when $t = h_D(x_0)$ is sufficient. Because of Lemma \ref{lowerESTim}, an appropriate estimate of the corresponding reflected solution $w := u-v$ is required, where $u$ satisfies \eqref{main_eqn}, and $v$ is the CGO solution of the Equation \eqref{CGOEQUation}. In this study, we verified that the reflected solution $w$ of Equation \eqref{gap_eqn} satisfies
\begin{equation}\label{starr}
\|w\|_{L^2(\Omega)} \leq C \| v\|_{L^1(D)},
\end{equation}
see Proposition \ref{lwb::prop1}.
To justify the enclosure method, we must construct the CGO solutions with an appropriate decay estimate in the correction term. We used the CGO solutions proposed in \cite{KRUPCHYK20121781}, and constructed CGO solutions of the form
\begin{equation}\label{CGOfirst}
v(x;h) = e^{\frac{\phi + i\psi}{h}} \left(a_0(x) + ha_1(x) + r(x,h)\right),
\end{equation}
where $\phi \in C^{\infty}(\wt {\Omega}, \mathbb{R})$ is a limiting Carleman weight for the semiclassical Laplacian on $\wt {\Omega}$, where $\Omega\subset\subset\wt {\Omega}$. Functions $a_0$ and $a_1$ are smooth, and the correction term $r$ satisfies $\norm{r}_{H_{scl}^{4}(\Omega)} = \mathcal{O}(h^2).$ 

The remainder of this paper is organized as follows. In Section \ref{mainresultsection}, we state our main results and discuss the CGO solutions that are useful in our proof. Section \ref{proofsection} provides proof of the main results. Finally, Section \ref{appendix} presents the existence and uniqueness of the boundary value problem for the fourth-order elliptic equation.
 \section*{Acknowledgement}
M. Kar was supported by the MATRICS grant (MTR/2019/001349) from SERB. 
G. Hwang was supported by the 2020 Yeungnam University research grant.

\section{Main result}\label{mainresultsection}
In this section, we present our main results. We first discuss the CGO solutions for the following fourth--order elliptic equation: 
\begin{align}\label{main eqn2}
\mathcal{L}_{q}v := (\Delta^2 + q(x))v = 0 \text{ in } \Omega,
\end{align}
where $ q \in L_{+}^\infty(\Omega)$. 
The CGO solutions of the form 
$$
v(x;h) = e^{\fr{\phi + i \psi}{h}}(a_0(x) + ha_1(x) + r(x;h)),
$$
were derived in \cite{KRUPCHYK20121781} to determine the first--order perturbation of the biharmonic operator.
Here, function $a_0$ solves the first transport equation
\begin{align}\label{trs_eqn1}
T^2a_0 = 0 \text{ in } \Omega,
\end{align}
where $T = (\nabla \phi + i \nabla \psi)\cdot \nabla + \fr12(\Delta\phi + i\Delta \psi)$.
In addition, function $a_1$ solves the second transport equation in $\Omega$:
\begin{align}\label{trs_eqn2}
T^2a_1 = -\fr12(\Delta \circ T + T \circ \Delta)a_0 + \frac{1}{4} A\cdot(i^{-1}\nabla\phi +\nabla\psi)a_0 \text{ in } \Omega.
\end{align}
 The Carleman weight is of the form
$$\phi(x) = \fr12 \log|x-x_0|^2$$
and
\begin{align}\label{psIII}
\psi(x) 
&= \fr{\pi}{2} - \tan^{-1} \fr{\omega\cdot(x-x_0)}{\sqrt{(x-x_0)^2 - (\omega\cdot(x-x_0))^2}} \\
&= \dist_{\mathbb{S}^{n-1}}\left(\fr{x-x_0}{|x-x_0|},\omega\right),
\end{align}
where $\omega \in \mathbb{S}^{n-1}$ is chosen such that $\psi$ is smooth near $\cj \Omega$, and $x_0$ is a fixed point outside the convex hull of $\Omega$.
In particular, in \cite{KRUPCHYK20121781}, they proved the following proposition.
\begin{prop}\label{SUMMerCGOpro}[Proposition 2.4 in \cite{KRUPCHYK20121781}]
Let $ q \in L_{+}^\infty(\Omega)$. Then, for an $h>0$ sufficiently small, there exist solutions $v(x;h) \in H^4(\Omega)$ to the equation:
\begin{equation}\label{praMCGO}
\Delta^2v + qv = 0 \ \text{in}\ \Omega
\end{equation}
of the form
$
v(x;h) = e^{\frac{\phi + i\psi}{h}}(a_0(x) + ha_1(x) + r(x;h)),
$
where $\phi$ is a limiting Carleman weight for the semi--classical Laplacian, and $\psi$ is defined as \eqref{psIII}. The amplitudes $a_0\in C^{\infty}(\overline{\Omega}) $ and $ a_1\in C^4(\overline{\Omega})$ satisfy Equations \eqref{trs_eqn1} and \eqref{trs_eqn2}, respectively, and the correction term $r$ satisfies
$\|r\|_{H^4_{scl}} = \mathcal{O}(h^2).$
\end{prop}
 Note that, for a given $h>0$ and $k\in \mathbb{N},$ the semi--classical norm of $r$ is defined as follows:
\[
\|r\|_{H_{scl}^{k}(\Omega)} := \left[ \sum_{|\alpha|\leq k} \int_{\Omega}|(hD)^{\alpha}u|^2 dx\right]^{1/2}.
\]
See [\cite{MR2952218}, Chapter 7] for an extensive study on these spaces and their properties.
Let $t$ be a constant and $h>0$ a small parameter. Maintaining the same notation as in Proposition \ref{SUMMerCGOpro}, we define 
\begin{equation}\label{cgo_So}
    v(x,h,t) = e^{\frac{1}{h}\left(t-\frac{1}{2}\log |x-x_0|^2\right)-\frac{i}{h}\psi(x)}(a_0(x) + ha_1(x) + r(x;h))
\end{equation}
to be a complex geometric optics solution with spherical phases for Equation \eqref{praMCGO}. 
 Using the CGO solutions with spherical phases, we define an indicator function as follows:
\begin{align}\label{indicatorFUN}
I_{x_0}(h, t) 
= \langle (\mathcal{N}_{D} - \mathcal{N}_{\emptyset})f,f \rangle 
= \int_{\doo\Omega} (\mathcal{N}_{D} - \mathcal{N}_{\emptyset})f\cdot \overline{f} dS,
\end{align}
where $dS$ denotes the surface measure of $\doo\Omega$. Here, $\mathcal{N}_{D}$ is the Dirichlet--to--Neumann map corresponding to the solution $u$ of problem \eqref{main_eqn}, and $\mathcal{N}_{\emptyset}$ denotes the Dirichlet--to--Neumann map corresponding to
a CGO solution with a spherical phase function, as described in Equation \eqref{cgo_So}, which satisfies
\begin{equation}\label{CGOEQUation}
\begin{aligned}
\left\{
\begin{array}{ll}
\Delta^2 v + v = 0 &\text{ in } \Omega\\
v = f_1 &\text{ on }\pa \Omega\\
\Delta v = f_2 &\text{ on }\pa \Omega.
\end{array}
\right.
\end{aligned}
\end{equation}
Note that $f$ is a vector--valued function defined as
$
f = (f_1, f_2) 
= \left( u|_{\pa \Omega}, (\Delta u)|_{\pa \Omega}\right) 
= \left(v|_{\pa \Omega}, (\Delta v)|_{\pa \Omega}\right),
$
such that the boundary values corresponding to \eqref{main_eqn} and \eqref{CGOEQUation} are the same. $\langle \mathcal{N}_{D}(f), f \rangle$ can be defined as follows:
\begin{align*}
 \langle \mathcal{N}_{D}(f), f \rangle 
&= \int_{\pa \Omega} \langle(\fr{\pa u}{\pa \nu}|_{\pa \Omega}, \fr{\pa}{\pa \nu} (\Delta u)|_{\pa \Omega}), (\cj f_1, \cj f_2) \rangle dS\\
&= \int_{\pa \Omega} \left(\fr{\pa u}{\pa \nu}\cj f_2 + \fr{\pa}{\pa \nu} (\Delta u)\cj f_1\right) dS.
\end{align*}

We introduce the distance function as 
$
h_D(x_0) := \inf_{x\in D} \frac{1}{2}\log |x-x_0|^2,
$ 
where $x_0 \in \mathbb{R}^3\setminus \overline{ch(\Omega)}$, and $ch(\Omega)$ denotes the convex hull of the domain $\Omega$. Note that, $e^{h_D(x_0)}$ measures the distance from $x_0$ to $D$.

We are now ready to formulate the main theorem of this study.
\begin{thm}\label{main thm}
Let $x_0 \in \mathbb{R}^3\setminus\overline{ch(\Omega)}$. Then there exist constants $c, \ C > 0$ independent of $h$ such that the indicator function $ I_{x_0}(h, t)$ satisfies 
\begin{equation}\label{eq_Impr}
    c\leq |h^{-3}\real I_{x_0}(h, h_D(x_0))| \leq \ C h^{-2}, \ h \ll 1.
\end{equation}
Moreover, it holds that
\begin{equation}\label{eq_gmrtt}
t-h_D{(x_0)} = \lim_{h \to 0} \frac{1}{2}h \log | \real I_{x_0}(h,t)|.
\end{equation}
\end{thm}
Using \eqref{eq_Impr} and \eqref{eq_gmrtt}, we can easily prove the following: 
\begin{enumerate}
\item When $t < h_D(x_0),$ we have
\[|h^{-3}\real I_{x_0}(h, t)| \leq Ce^{-\frac{c}{h}}, \ h \ll 1.\]
In particular, $\lim_{h \to 0} |\real I_{x_0}(h, t)| = 0$.
\item When $t > h_D(x_0),$ we have
\[|h^{-3}\real I_{x_0}(h, t)| \geq C e^{\frac{c}{h}}, \ h \ll 1.\]
In particular, $\lim_{h \to 0} |\real I_{x_0}(h, t)| = \infty$.
\end{enumerate}

From this theorem, we obtain the certain asymptotic behavior of the indicator function required to reconstruct the unknown obstacle from the boundary data. Specifically, let us fix a point $x_0 \in \mathbb{R}^3\setminus\overline{ch(\Omega)}.$ Then, we observe that the complex geometric optics solution (see Proposition \ref{SUMMerCGOpro}) exhibits an asymptotic behavior, that is, it grows exponentially inside the sphere
$
S = \{x\in \mathbb{R}^{2}  ; |x-x_0|= e^{t} \}
$
for a sufficiently small $h>0$ and decays exponentially fast outside the sphere. Using this feature of the CGO solution, we can observe that, when $t<h_D(x_0)$, the indicator function $I_{x_0}(h,t)$ vanishes exponentially for a sufficiently small $h>0.$ Now, we can expand the sphere such that when time $t\geq h_D(x_0)$, the obstacle intersects the sphere, and by Theorem \ref{main thm}, the indicator function becomes large for a small $h$. Finally, moving point $x_0$ around $\overline{ch(\Omega)}$, we can enclose the unknown obstacle using the spheres. In this manner, we can recover the convex hull of the obstacle and some of its non--convex part.  

\section{Proof of the Theorem \ref{main thm}}\label{proofsection}
In this section, we provide a proof of Theorem \ref{main thm}.
By the definition of the indicator function and complex geometric optics solutions, it follows that
\begin{equation}\label{Pf::1}
  I_{x_0}(h, t) = e^{\frac{2}{h}(t-h_D(x_0))}I_{x_0}(h, h_D(x_0)).  
\end{equation}
Because of \eqref{Pf::1}, proving Equation \eqref{eq_Impr} in Theorem \ref{main thm} is sufficient. Let us recall the integration by parts formula, which is often useful in estimates. For any $\phi\in H^4(\Omega)$
and $\psi\in H^2(\Omega),$ Green's theorem holds as follows:
\begin{align}\label{INPART}
 \int_\Omega \nabla \cdot \nabla (\Delta \phi) \overline{\psi} dx \nonumber 
&= - \int_{\Omega} \nabla( \Delta \phi)\cdot \overline{\nabla \psi} dx + \int_{\pa \Omega} \fr{\pa}{\pa \nu} ( \Delta \phi)\overline{\psi} dS \nonumber \\
&= \int_{\Omega} ( \Delta \phi)\overline{\Delta\psi} dx - \int_{\pa \Omega} \Delta \phi \overline{\fr{\pa\psi}{\pa \nu}}dS + \int_{\pa \Omega} \fr{\pa}{\pa \nu} ( \Delta \phi)\overline{\psi} dS.
\end{align}

\subsection{Lower and upper bound of $I_{x_0} (h, h_D(x_0))$}
Let $v$ be a CGO solution of the biharmonic equation
\begin{equation}\label{modi_eqn}
\begin{aligned}
\left\{
\begin{array}{ll}
\Delta^2 v + v = 0 &\text{ in } \Omega\\
v = f_1 &\text{ on }\pa \Omega\\
\Delta v = f_2 &\text{ on }\pa \Omega.
\end{array}
\right.
\end{aligned}
\end{equation}
Let $w := u-v$ be the reflected solution, where $u$ is the solution to problem \eqref{main_eqn}. Then, $w$ satisfies the following boundary value problem:
\begin{equation}\label{gap_eqn}
\begin{aligned}
\left\{
\begin{array}{ll}
\Delta^2 w  +  \wt n(x)w = 
 %\quad\quad\quad\quad\quad
 - (\wt n - 1)v &\text{ in } \Omega\\
w = 0 &\text{ on } \pa \Omega\\
\Delta w = 0 &\text{ on } \pa \Omega.
\end{array}
\right.
\end{aligned}
\end{equation}

The main step in proving the lower and upper bounds of $I_{x_0}(h, h_D(x_0))$ is to prove the following proposition.
%and we follow an argument similar to the proof of Lemma 3.1 in \cite{Ikehata_2022}.
\begin{prop}\label{lwb::prop1}
Let $\Omega$ be a smooth domain in $\mathbb{R}^3$ and the inclusion $D$ to be strictly embedded inside $\Omega$. Then, there exists $C>0$ such that
\[\|w\|_{L^2(\Omega)}\leq C \| v\|_{L^1(D)}.\]
\end{prop}

\begin{proof}
Let us define a function space $$ X:= \left\{ \phi \in H^4(\Omega); \phi = \Delta \phi = 0\text{ on }\pa \Omega\right\}.$$
Suppose $\Phi \in X$ is a weak solution of the equation
\begin{equation}\label{cond_Phi}
\begin{aligned}
\left\{
\begin{array}{ll}
\Delta^2 \Phi  +  \wt n \Phi = w & \text{ in } \Omega\\
\Phi = 0 &\text{ on } \pa \Omega\\
\Delta \Phi = 0 &\text{ on } \pa \Omega,
\end{array}
\right.
\end{aligned}
\end{equation}
where $w$ satisfies Equation \eqref{gap_eqn}.
By multiplying Equation \eqref{cond_Phi} by $\overline{w}$ and integrating by parts, we obtain,
\begin{align*}
& \int_\Omega |w(x)|^2dx 
 = \int_\Omega \nabla \cdot \nabla ( \Delta \Phi)\overline{w(x)}dx + \int_\Omega \wt n(x) \Phi(x) \overline{w(x)} dx\\
& \qquad= \int_{\Omega} \Delta \Phi \overline{\Delta w} dx- \int_{\pa \Omega} \Delta\Phi \overline{\fr{\pa w}{\pa \nu}}dS +  \int_{\doo\Omega}\frac{\doo}{\doo\nu}(\Delta\Phi)\overline{w(x)}dS \\ 
& \qquad \qquad + \int_\Omega \wt n(x) \Phi(x) \overline{w(x)}dx.
\end{align*}
Because $w=0, \Delta \Phi = 0$ on $\pa \Omega$, the above identity becomes
\begin{align}
\int_\Omega |w(x)|^2dx  = \int_{\Omega} \Delta \Phi \overline{\Delta w} dx + \int_\Omega \wt n(x) \Phi(x) \overline{w(x)}dx. \label{cons}
\end{align}
By multiplying Equation \eqref{gap_eqn} by $\Phi$ and integrating by parts, we obtain
\begin{equation}\label{cons2}
\begin{aligned}
\int_\Omega  \Delta \overline{w}\Delta \Phi  dx +  \int_\Omega \wt n(x)\overline{ w(x)}\Phi(x) dx  =-\int_{\Omega} (\wt n -1) \overline{v}  \Phi dx. 
\end{aligned}
\end{equation}
Then, by combining the real parts of Equation \eqref{cons} and \eqref{cons2}, we obtain
%\[\|w\|_{L^2(\Omega)}^2 = - \int_\Omega [\chi_D \ga_D \Delta v \Delta \phi +  \chi_D q_D \phi].\]
\begin{align*}
\|w\|_{L^2(\Omega)}^2 
&=  - \real\int_D n_D \overline{v} \Phi dx. 
\end{align*}
Using the Cauchy-Schwarz inequality, we obtain
\begin{equation}\label{cons3}
\begin{aligned}
 \|w\|_{L^2(\Omega)}^2 
&\leq C  \|v\|_{L^1(D)}\|\Phi\|_{L^{\infty}(D)}.
\end{aligned}
\end{equation}
Now, we apply the Sobolev embedding and elliptic estimate (see Lemma \ref{regularityAPP}) to obtain
\begin{equation}\label{cons4}
\begin{aligned}
\|\Phi\|_{L^{\infty}(D)} \leq \|\Phi\|_{L^{\infty}(\Omega)}
\leq C\|\Phi\|_{H^{2}(\Omega)} 
  \leq C\|w\|_{L^{2}(\Omega)}.
\end{aligned}
\end{equation}
Finally, the conclusion follows from Equations \eqref{cons3} and \eqref{cons4}.
\end{proof}

\begin{lem}\label{lowerESTim}
Assume that functions $v$ and $w$ are the solutions to Equations \eqref{modi_eqn} and \eqref{gap_eqn}, respectively. Then, we obtain:
\begin{itemize}
\item[(1)] the lower bound of the indicator function
\begin{align*}
|\real I_{x_0} (h, t)| 
&\geq C\int_D |v(x)|^2 dx - c\int_D |w(x)|^2 dx ,
\end{align*}
\item [(2)] and the upper bound of the form
\[
|\real I_{x_0} (h, t)|  \leq C\int_D |v(x)|^2 dx + c\int_D |w(x)|^2 dx,
\]
\end{itemize}
where $C $ and $ c>0$ are constants.
\end{lem}
\begin{proof}
Let us recall that $\mathcal{N}_{D}$ denotes the Dirichlet--to--Neumann map that encodes the current measurement on the boundary $\partial\Omega$ corresponding to the boundary voltage $u = f$ prescribed on $\partial\Omega$, when an obstacle $D$ is embedded in the domain $\Omega.$ We write the weak form of $\mathcal{N}_{D}$ as:
\begin{equation}\label{weak_id}
\begin{aligned}
 \langle \mathcal{N}_{D}f, f \rangle
&= \int_{\pa \Omega} \langle (\fr{\pa u}{\pa \nu}, \fr{\pa}{\pa \nu}(\Delta u)), \overline{(f_1, f_2)} \rangle dS\\
&= \int_{\pa \Omega} \left(\fr{\pa u}{\pa \nu}\overline{f_2} + \fr{\pa}{\pa \nu}(\Delta u)\overline{f_1}\right) dS,
\end{aligned}
\end{equation}
where $u$ satisfies Eqaution \eqref{main_eqn}. Moreover, we denote $\mathcal{N}_{\emptyset}$ as the Dirichlet--to--Neumann map when no obstacle is in $\Omega$. It has the following weak form:
\[
\langle \mathcal{N}_{\emptyset} f, f \rangle = \int_{\pa \Omega} \left(\fr{\pa v}{\pa \nu}\overline{f_2} + \fr{\pa}{\pa \nu}(\Delta v)\overline{f_1}\right),
\]
where $v$ satisfies Equation \eqref{modi_eqn}.
By multiplying problem \eqref{main_eqn} by $\overline{v}$ and integrating by parts, we obtain
\begin{align}
0 &= \int_\Omega \nabla \cdot \nabla ( \Delta u)\overline{v} dx + \int_{\Omega} \wt n u\overline{v} dx  \nonumber \\
&= \int_\Omega  \Delta u \Delta \overline{v} dx - \int_{\pa \Omega}f_2 \fr{\pa \overline{v}}{\pa \nu} dS + \int_{\pa \Omega} \fr{\pa}{\pa \nu} (\Delta u)\overline{f_1} dS +  \int_{\Omega} \wt n u\overline{v} dx 
\label{barvmu}.
\end{align}
Using Equation \eqref{barvmu}, the Dirichle--to--Neumann map can be written as
\begin{align}
\langle \mathcal{N}_{D}f, f \rangle
&= \int_{\pa \Omega} \left(\fr{\pa u}{\pa \nu}\overline{f_2} + \fr{\pa\overline{v}}{\pa \nu}f_2\right) dS - \int_{\Omega}\Delta u \Delta\overline{v} dx  - \int_{\Omega}\wt{n}u\overline{v} dx. \label{dnmagna}
\end{align}
Moreover, by multiplying
Equation \eqref{modi_eqn} by $\overline{u}$ and integrating by parts, we obtain 
\begin{align}
0&= \int_{\Omega} (\Delta^2 v)\overline{u} dx +  \int_{\Omega}\overline{u}v dx \nonumber\\
&= \int_\Omega \Delta \overline{u }\Delta v dx - \int_{\pa \Omega} f_2 \fr{\pa \overline{u}}{\pa \nu} dS + \int_{\pa \Omega} \fr{\pa}{\pa \nu}(\Delta v)\overline{f_1} dS +  \int_\Omega \overline{u}v dx \label{mulbaru}.
\end{align}
By taking the real part of Equation \eqref{mulbaru}, we compute the following Dirichlet--to--Neumann map:
\begin{align}
\real \langle \mathcal{N}_{\emptyset}f, f \rangle
& = \real \int_{\pa\Omega} \frac{\pa v}{\pa\nu} \overline{f_2} dS+  \real \int_{\pa\Omega} \frac{\pa (\Delta \overline{v})}{\pa\nu} f_1 dS  \nonumber \\
& = \real \int_{\pa\Omega} \frac{\pa v}{\pa\nu} \overline{f_2} dS - \real \int_\Omega \Delta u \Delta \overline{v}  dx \nonumber 
\\
& \qquad \qquad + \real  \int_{\pa \Omega}\overline{f_2}\fr{\pa u}{\pa \nu} dS -  \real \int_{\Omega} u\overline{v} dx. \label{dnmap101}
\end{align}
Then, Equations \eqref{dnmagna} and \eqref{dnmap101} provide that
\begin{align}
-\real I_{x_0}(h, t) &= -\real \int_{\pa \Omega} \langle (\mathcal{N}_{D} - \mathcal{N}_{\emptyset})f,f \rangle dS \nonumber\\
&= \real \int_{\Omega} (\wt{n}-1)u\overline{v} dx.  \label{indiman}
\end{align}
By applying the Cauchy-Schwartz inequality to Equation \eqref{indiman}, we obtain the upper bound of $-I_{x_0}(h,t)$.

To estimate the lower bound of $-\real I_{x_0} (h,t),$ we use Cauchy's $\epsilon$ inequality (see [\cite{Evans:1998}, Appendix]), as follows:  
\begin{align*}
-\real I_{x_0}(h,t)
=  \real \int_{\Omega} (\wt{n}-1)u\overline{v} dx
\geq C\int_D|v|^2dx  - c\int_{D}|w|^2dx,
\end{align*}
where $C $ and $ c>0$ are positive constants.
Finally, conclusions can be derived.
%\begin{align*}
%-I_{x_0}(h, t) 
%& \geq C_1 \int_D |\Delta v(x)|^2dx - C_2 \int_\Omega |\nabla w(x)|^2dx - C_3 \int_\Omega |w(x)|^2dx \\
%& \quad - C_4\int_D |\nabla v(x)|^2dx - C_5\int_D |v(x)|^2dx,
%\end{align*}
%where, $C_i>0$ are positive constants.
\end{proof}

\subsection{End of the proof of Theorem \ref{main thm}}

In this part, we follow \cite{MR2367868} for the $C^2$ regular unknown obstacle.
We provide a detailed estimate for the complex geometrical optics solutions related to the bi-Laplace equation.
Let $B(\alpha, \delta)$ denote a ball of radius $\delta$ centered at $\al$. We define,
\[
K := \pa D \cap \{x\in \mathbb{R}^3 ; \frac{1}{2}\log|x-x_0|^2 = h_D(x_0)\},
\]
where $x_0 \in \mathbb{R}^3\setminus \overline{ch({\Omega})}.$ The set $K$ can be covered by open covers, i.e., $K \subset \cup_{\al \in K}B(\al, \del)$. Because $K$ is compact, $\al_1, \cdots, \al_N \in K$ exist, such that $K \subset \cup_{j=1}^N B(\al_j, \delta)$. Then, we define $D_{j, \delta} := D \cap B(\al_j, \delta)$ and $D_\delta = \cup_{j=1}^N D_{j, \delta}$. Note that,
$$
\int_{D\setminus \overline{D_\delta}} e^{-\frac{p}{h}(\frac{1}{2}\log|x-x_0|^2-h_D(x_0))}dx = \mathcal{O}(e^{-\frac{pc}{h}}) \text{ as } h \to 0.
$$
We introduce a change of co-ordinates as in \cite{MR2367868},
\[
y' = x', \qquad
y_3 = \frac{1}{2}\log|x-x_0|^2 - h_D(x_0),
\]
where $x' = (x_1, x_{2}), \ y'=(y_1, y_2), \ x = (x', x_3), $ and $ y= (y',y_3)$.
By the $C^2$ assumption on $D$, there exist positive constants $K_1, K_2$ such that
\[
K_1\left|y^{\prime}\right|^2 \leq l_j\left(y^{\prime}\right) \leq K_2\left|y^{\prime}\right|^2
\]
where $l_j\left(y^{\prime}\right)$ is the parametrization of $\partial D$ near $\alpha_j$.
Then, we have the following lemma.

\begin{lem}\label{estiMATEs}
The following upper and lower estimates hold for $h \ll 1$:
\begin{enumerate}
\item For $1 \leq q < 2$, we have 
 \begin{align*}
\int_D |v(x)|^qdx \leq &C h \left(\sum_{j=1}^N \int_{|y'| < \delta} e^{-\frac{q l_j(y')}{h}} dy' + \left(\sum_{j=1}^N \int_{|y'| < \delta} e^{-\frac{2q l_j(y')}{(2-q)h}} dy'\right)^{\frac{2-q}{2}}\right)\\
& + \text{exponentially decaying terms}.
 \end{align*}
\item When $q=2$, it follows that
\begin{align*}
&\int_D |v(x)|^2dx \leq C h \sum_{j=1}^N \int_{|y'| < \delta} e^{-\frac{2 l_j(y')}{h}} dy' + \text{exponentially decaying terms}\\
&\text{and}\\
&\int_D |v(x)|^2dx \geq C h \sum_{j=1}^N \int_{|y'| < \delta} e^{-\frac{2 l_j(y')}{h}} dy' + \text{exponentially decaying terms}.
\end{align*}
\end{enumerate}
\end{lem} 

\begin{proof}
Recall that, when $t= h_D(x_0),$ the complex geometrical optics solution $v$ is of the form
$
v(x,h) = e^{\fr{\phi + i\psi}{h}} (a_0(x) + ha_1(x) + r(x,h))
$
where $\phi = h_D(x_0) - \frac{1}{2}\log|x-x_0|^2.$ 
In addition, the correction term $r$ satisfies
\[\|r\|_{H^4_{scl}(\Omega)} \leq \mathcal{O}(h^2).\]
Simplifying this, we obtain
\begin{equation}
\|r\|_{L^2(\Omega)} \leq h^2.
\end{equation}
(1) We first compute the following integral using the H\"older inequality. 
\begin{align*}
& \int_D |v(x)|^q dx\\
&\leq C \int_D e^{-\fr{q}{h}(\frac{1}{2}\log|x-x_0|^2 - h_D(x_0))} (a_0^q + h^qa_1^q + r^q) dx\\
&= \left(\int_{D_\delta} + \int_{D\setminus D_\delta}\right) e^{-\fr{q}{h}(\frac{1}{2}\log|x-x_0|^2 - h_D(x_0))} (a_0^q + h^qa_1^q  + r^q) dx\\
&\leq C(1+h^q) \sum_{j=1}^N \int_{|y'| < \delta} dy' \int_{l_j(y')}^\delta e^{-\fr{qy_n}{h}}dy_n\\
&\quad + C\left(\int_{D_\delta} e^{-\frac{qp}{h}(\frac{1}{2}\log|x-x_0|^2 - h_D(x_0))}dx\right)^{\fr{1}{p}}\left(\int_{D_\delta}r^2\right)^{\fr{q}{2}} + Ce^{-\frac{qc}{h}},
% &\leq C(1+h^q)\Big[h\sum_{j=1}^N \int_{|y'| < \delta} e^{-\fr{ql_j(y')}{h}} dy' - \fr{ch}{q}e^{-\fr{q\delta}{h}}\Big]
% &\quad + ch^{2q} \Big[h \sum_{j=1}^N \int_{|y'| < \delta} e^{-\fr{qpl_j(y')}{h}} dy'\Big]^{\fr{1}{p}} + ch^{2q}h^{\fr{1}{p}}e^{-\fr{q\delta} {h}} + Ce^{-\frac{qc}{h}}.
\end{align*}
where $p = \fr{2}{2-q}$. Then, the estimate $\|r\|_{L^2(\Omega)} \leq h^2$ yields
\begin{align*}
& \int_D |v(x)|^q dx \\
&\leq C(h+h^{q+1})\sum_{j=1}^N \int_{|y'| < \delta} e^{-\fr{ql_j(y')}{h}} dy'+ h^{2q + \fr{2-q}{2}}\left(\sum_{j=1}^N \int_{|y'| < \delta} e^{-\fr{qpl_j(y')}{h}} dy'\right)^{\frac{1}{p}}\\ 
&\quad - C\fr{(1+h^q)h}{q}e^{-\fr{q\delta}{h}} + Ch^{2q+\fr{2-q}{2}}e^{-\fr{q\delta}{h}} + Ce^{-\frac{qc}{h}}.
\end{align*}
Here, we observe that\footnote{Here, $o(\cdot)$ denotes small $o$ notation.}
\[h^{q+1} \leq o(h) \text{ and } h^{2q + \fr{2-q}{2}} \leq o(h)\]
for a sufficiently small $h$, and the last three terms are exponentially decaying. Therefore, it follows that
 \begin{align*}
\int_D |v(x)|^qdx \leq &C h \left(\sum_{j=1}^N \int_{|y'| < \delta} e^{-\frac{q l_j(y')}{h}} dy' + \left(\sum_{j=1}^N \int_{|y'| < \delta} e^{-\frac{qp l_j(y')}{h}} dy'\right)^{\frac{1}{p}}\right)\\
& + \text{exponentially decaying terms}.
 \end{align*}
(2) We then compute the upper bound estimate of the $L^2$-norm of $v$.
\begin{align*}
&\int_{D} |v(x)|^2 dx \\
&\leq C \int_D e^{-\fr{2}{h}(\frac{1}{2}\log|x-x_0|^2 - h_D(x_0))}dx + C h^2\int_D e^{-\fr{2}{h}(\frac{1}{2}\log|x-x_0|^2 - h_D(x_0))}dx\\
&\quad + C \int_D e^{-\fr{2}{h}(\frac{1}{2}\log|x-x_0|^2 - h_D(x_0))} r^2dx\\
&\quad := I_1 + I_2 + I_3.
\end{align*}
For $I_1$ and $I_2$, we have
\begin{align*}
I_1 &:= \int_D e^{-\fr{2}{h}(\frac{1}{2}\log|x-x_0|^2 - h_D(x_0))}dx\\
&\leq \left(\int_{D_\delta} + \int_{D\setminus D_\delta}\right) e^{-\fr{2}{h}(\frac{1}{2}\log|x-x_0|^2 - h_D(x_0))}dx\\
&\leq \int_{D_\delta} e^{-\fr{2}{h}(\frac{1}{2}\log|x-x_0|^2 - h_D(x_0))}dx + \text{exponentially decaying terms}\\
&\leq C \sum_{j=1}^N \int_{|y'|< \delta} dy' \int_{l_j(y')}^{\delta} e^{-\fr{2y_n}{h}}dy_n + \text{exponentially decaying terms} \\
&\leq Ch \sum_{j=1}^N \int_{|y'| < \delta} e^{-\fr{2l_j(y')}{h}} dy'  + \text{exponentially decaying terms},
\end{align*}
and
\begin{align*}
I_2 &:= h^2 \int_D e^{-\fr{2}{h}(\frac{1}{2}\log|x-x_0|^2 - h_D(x_0))}dx\\
&= h^2 \left(\int_{D_\delta} + \int_{D\setminus D_\delta}\right) e^{-\fr{2}{h}(\frac{1}{2}\log|x-x_0|^2 - h_D(x_0))}dx\\
&\leq ch^3 \sum_{j=1}^N \int_{|y'|<\delta} e^{-\fr{2l_j(y')}{h}}dy' + \text{exponentially decaying terms}.
\end{align*}
Now, we estimate the remainder.
The H\"older inequality yields
\begin{align*}
I_3 &:= \int_D e^{-\fr{2}{h}(\frac{1}{2}\log|x-x_0|^2 - h_D(x_0))}r^2 dx\\
& \leq C\|r\|_{L^2(\Omega)}^2 \leq Ch^4.
\end{align*}
Furthermore, we observe that for $h \ll 1$,
\begin{align*}
\int_{|y'| < \delta} e^{-\fr{2l_j(y')}{h}dy'} & \geq C \sum_{j=1}^N \int_{|y'| < \delta} e^{-\fr{2|y'|}{h}dy'} \\
& \geq Ch^{2} \sum_{j=1}^N \int_{|y'| < \frac{\delta}{h}} e^{-2|y'|}dy'\\
& \geq Ch^{2}.
\end{align*}
Here, we use $l_j(y') \leq K_2 |y'|$ if $\partial D$ is $C^{2}$. Hence, we obtain 
\begin{align*}
I_3 \leq Ch^2\int_{|y'| < \delta} e^{-\fr{2l_j(y')}{h}dy'}.
\end{align*}
Because the first term $I_1$ dominates the remaining terms, it follows that
\[\int_D |v(x)|^2dx \leq C h \sum_{j=1}^N \int_{|y'| < \delta} e^{-\fr{2l_j(y')}{h}dy'} + \text{exponentially decaying terms}.\]

For the lower bound estimate of the $L^2$-norm of $v$, we observe that 
\begin{align*}
&\int_{D} |v(x)|^2 dx \\
&\geq C \int_D e^{-\fr{2}{h}(\frac{1}{2}\log|x-x_0|^2 - h_D(x_0))}dx - C h^2\int_D e^{-\fr{2}{h}(\frac{1}{2}\log|x-x_0|^2 - h_D(x_0))}dx\\
&\quad - C \int_D e^{-\fr{2}{h}(\frac{1}{2}\log|x-x_0|^2 - h_D(x_0))} r^2dx\\
& = I_1 - I_2 - I_3.
\end{align*}
For $I_1$, we have
\begin{align*}
I_1 &= \int_D e^{-\fr{2}{h}(\frac{1}{2}\log|x-x_0|^2 - h_D(x_0))}dx \\
&\geq \int_{D_\delta} e^{-\fr{2}{h}(\frac{1}{2}\log|x-x_0|^2 - h_D(x_0))}dx\\
&\geq C \sum_{j=1}^N \int_{|y'|< \delta} dy' \int_{l_j(y')}^{\delta} e^{-\fr{2y_n}{h}}dy_n \\
& \geq Ch \sum_{j=1}^N \int_{|y'| < \delta} e^{-\fr{2l_j(y')}{h}} dy'  - \fr{C}{2}he^{-\fr{2\delta}{h}}.
\end{align*}
Therefore, with the previous estimates for $I_2$ and $I_3$, we conclude that
\[\int_D |v(x)|^2dx \geq C h \sum_{j=1}^N \int_{|y'| < \delta} e^{-\fr{2l_j(y')}{h}dy'} + \text{exponentially decaying terms}.\]

\end{proof}

\begin{proof}[Proof of Theorem \ref{main thm}]
We first prove 
\[ \fr{\|w\|_{L^2(D)}^2}{\| v\|_{L^2(D)}^2} \leq Ch, \ \ h \ll 1.\] 
Proposition \ref{lwb::prop1} gives that
\begin{align*}
   \fr{\|w\|_{L^2(D)}^2}{\| v\|_{L^2(D)}^2} \leq C \fr{\|v\|_{L^1(D)}^2}{\| v\|_{L^2(D)}^2}. 
\end{align*}
Using Lemma \ref{estiMATEs} with an elementary inequality
\[
\sum_{j=1}^N \int_{|y'| < \delta} e^{-\fr{l_j(y')}{h}dy'} 
\leq C\left(\sum_{j=1}^N \int_{|y'| < \delta} e^{-\fr{2l_j(y')}{h}dy'}\right)^{1/2},
\]
we get
\begin{align*}
&\frac{\left(\int_{D}|v|dx\right)^2}{\int_D |v|^2 dx} \\
&\leq C \frac{h^2\left( \sum_{j=1}^N \int_{|y'| < \delta} e^{-\fr{l_j(y')}{h}dy'}\right)^2 + \text{exponentially decaying terms}}{h \sum_{j=1}^N \int_{|y'| < \delta} e^{-\fr{2l_j(y')}{h}dy'} + \text{exponentially decaying terms}} \\
& \leq C h 
\frac{ \sum_{j=1}^N \int_{|y'| < \delta} e^{-\fr{2l_j(y')}{h}dy'} + \text{exponentially decaying terms}}{ \sum_{j=1}^N \int_{|y'| < \delta} e^{-\fr{2l_j(y')}{h}dy'} + \text{exponentially decaying terms}} \\
& = \mathcal{O}(h) \qquad (h\to 0). 
\end{align*}
Therefore, from Lemma \ref{lowerESTim} (1), we obtain
\begin{align*}
    \fr{|\real I_{x_0}(h,h_D(x_0))|}{\int_D|v|^2dx} 
    \geq C - c \fr{\int_D|w|^2dx}{\int_D|v|^2dx} 
    \geq C - ch 
    \geq C, \qquad (h\to 0).
\end{align*}
Using Lemma \ref{estiMATEs} (2), and 
\begin{align*}
\sum_{j=1}^N\int_{|y'| < \delta} e^{-\fr{2l_j(y')}{h}dy'} \geq  Ch^2,
\end{align*}
we have
\[\int_D |v(x)|^2dx \geq Ch^{3}, \ \ h \ll 1. \]
Therefore, it follows that
\[
|\real I_{x_0}(h,t)| \geq C \int_D|v(x)|^2 dx \geq Ch^3,
\]
which yields
\[
|h^{-3}\real I_{x_0}(h,t)| \geq C >0 \qquad \ \text{for}\ h\ll 1.
\]

For the upper bound, we use Lemma \ref{lowerESTim} (2), Proposition \ref{lwb::prop1}, and Lemma \ref{estiMATEs} (1) as follows:
\begin{align*}
    |\real I_{x_0}(h,t)| 
    & \leq C\int_{D}|v(x)|^2 dx + C\int_{D}|w(x)|^2 dx \\
    & \leq C\int_{D}|v(x)|^2 dx + \left( \int_{D}|v(x)| dx\right)^2 \\
    & \leq Ch \sum_{j=1}^N \int_{|y'| < \delta} e^{-\frac{2 l_j(y')}{h}} dy'\\
    &+ h^2 \left(\sum_{j=1}^N \int_{|y'| < \delta} e^{-\frac{ l_j(y')}{h}} dy' + \left(\sum_{j=1}^N \int_{|y'| < \delta} e^{-\frac{2 l_j(y')}{h}} dy'\right)^{\fr{1}{2}}\right)^2\\
& + \text{exponentially decaying terms}\\
& \leq Ch.
\end{align*}

\end{proof}

\section{Appendix}\label{appendix}
We provide a detailed proof of the $L^2$ regularity estimate of the solutions for the bi-Laplace equation with non-smooth coefficients in the $n$-dimensional domain.
\begin{lem}\label{regularityAPP}
We assume $\Omega\subset\mathbb{R}^n, n\geq 3$ to be an open bounded set with a sufficiently smooth regular boundary. Let $u$ be a solution of the following fourth--order elliptic equation:
\begin{equation}\label{appregulll}
\begin{aligned}
\left\{
\begin{array}{ll}
\Delta^2 u  +  \tilde{n}(x)u = f &\text{ in } \Omega\\
u = h_1 &\text{ on } \pa \Omega\\
\Delta u = h_2 &\text{ on } \pa \Omega,
\end{array}
\right.
\end{aligned}
\end{equation}
where coefficient $\tilde{n}(x) \in L^{\infty}_{+}(\Omega)$.
Then, for any $f\in L^2(\Omega)$, $ h_1 \in H^{7/2}(\doo\Omega)$ and $h_2 \in H^{3/2}(\doo\Omega)$, there exists a unique solution $u\in H^4(\Omega)$ to Equation \eqref{appregulll} such that
\[
\norm{u}_{H^4(\Omega)} \leq C \left[\norm{f}_{L^2(\Omega)} + \norm{h_1}_{H^{7/2}(\doo\Omega)} + \norm{h_2}_{H^{3/2}(\doo\Omega)}\right],
\]
where $C>0$ is a constant independent of the data.
\end{lem}

\begin{proof}
Define the Sobolev space 
\[
H^2(\Omega) := \{u\in L^2(\Omega) ; D^{\alpha}u \in L^2(\Omega), \ \text{for}\ |\alpha|\leq 2\}
\]
with the usual Sobolev norm
\[
\|u\|_{H^2(\Omega)} = \sum_{|\alpha|=2}\|D^{\alpha}u\|_{L^2(\Omega)}. 
\]
We now define another norm
\[
|||u||| := \|\Delta u\|_{L^2(\Omega)}. 
\]
Notice that these two norms $\|\cdot\|$ and $|||\cdot|||$ are equivalent for all $u\in H^2(\Omega)\cap H_{0}^{1}(\Omega)$. It follows by combining interpolation inequalities (see [\cite{Evans:1998}, Section 5.10, Problem 9])
\[
\|\nabla u\|_{L^2(\Omega)}^{2} \leq C \| u\|_{L^2(\Omega)}\|D^2 u\|_{L^2(\Omega)},
\]
the Poincar\'e inequality $\|u\|_{L^2(\Omega)} \leq \|\nabla u\|_{L^2(\Omega)}$, and the identity $$\|D^2 u\|_{L^2(\Omega)}=\|\Delta u\|_{L^2(\Omega)},$$ (see [\cite{Gilbarg:Trudinger:1983}, Corollary 9.10]). Therefore, the space $H_{0}^{2}(\Omega)$ can be defined as the closure of $C_{c}^{\infty}(\Omega)$ with respect to the norm $|||\cdot|||$. See [\cite{Gazzola}, Chapter 2] for more details. We now consider the following homogeneous boundary value problem
\begin{equation}\label{homo}
\begin{aligned}
\left\{
\begin{array}{ll}
\Delta^2 u  +  \tilde{n}(x)u = f &\text{ in } \Omega\\
u = 0 &\text{ on } \pa \Omega\\
\Delta u = 0 &\text{ on } \pa \Omega.
\end{array}
\right.
\end{aligned}
\end{equation}
We then define the bilinear form 
\[
\mathcal{L} : H^2(\Omega)\cap H_{0}^{1}(\Omega) \times H^2(\Omega)\cap H_{0}^{1}(\Omega) \rightarrow \mathbb{R}
\]
by 
\[
\mathcal{L} (u,v) := \langle u,v \rangle = \int_{\Omega}\Delta u \Delta v dx + \int_{\Omega}\tilde{n}(x)u v dx.
\]
It is easy to verify that $\mathcal{L}$ is bounded $$|\mathcal{L} (u,v)| \leq C \|\Delta u\|_{L^2(\Omega)}\|\Delta v\|_{L^2(\Omega)}$$ as well as coercive: $|\mathcal{L} (u,u)| \geq C\|\Delta u\|_{L^2(\Omega)}^2$ for all $u, v \in H^2(\Omega)\cap H_{0}^{1}(\Omega)$. By the Lax-Milgram lemma, for any $f\in L^2(\Omega)$, there is a unique weak solution $u\in H^2(\Omega)\cap H_{0}^{1}(\Omega)$ to Equation \eqref{homo} such that 
\[
\norm{u}_{H^2(\Omega)} \leq C \norm{f}_{L^2(\Omega)},
\]
where $C>0$ is a constant.
To obtain a strong solution to Equation \eqref{homo}, the interior and boundary regularity results can be used \cite{Grubb}. In this case, we obtain an $H^4(\Omega)$ solution to Equation \eqref{homo} and estimate of the form
\[
\norm{u}_{H^4(\Omega)} \leq C \norm{f}_{L^2(\Omega)}.
\]
To prove the well--posedness of Equation \eqref{appregulll}, we first reduce the problem to a homogeneous form. Let us define $v := u - \tilde{h}_1$ such that $\tilde{h}_1|_{\partial\Omega} = h_1$ and $\tilde{h}_2|_{\partial\Omega} = h_2$. Then, Equation \eqref{appregulll} can be reduced to
\begin{equation}\label{homo2}
\begin{aligned}
\left\{
\begin{array}{ll}
\Delta^2 v  +  \tilde{n}(x)v = f - \tilde{n}\tilde{h}_1 - \Delta \tilde{h}_2&\text{ in } \Omega\\
v = 0 &\text{ on } \pa \Omega\\
\Delta v = 0 &\text{ on } \pa \Omega.
\end{array}
\right.
\end{aligned}
\end{equation}
Because the trace maps $T_{1} : H^4(\Omega) \rightarrow H^{7/2}(\Omega)$ by $T_{1}(u) = u|_{\partial\Omega}$
and $T_{2} : H^2(\Omega) \rightarrow H^{3/2}(\Omega)$ by $T_{2}(\Delta u) = (\Delta u)|_{\partial\Omega}$
are surjective and have bounded inverses, the functions $\tilde{h}_1 \in H^4(\Omega)$ and $\Delta\tilde{h}_2 \in H^2(\Omega)$, respectively. Therefore, the $H^4(\Omega)$ estimates for the homogeneous Equation \eqref{homo} imply that 
\[
\norm{u}_{H^4(\Omega)} \leq C \left[\norm{f}_{L^2(\Omega)} + \norm{h_1}_{H^{7/2}(\doo\Omega)} + \norm{h_2}_{H^{3/2}(\doo\Omega)}\right].
\]

\end{proof}

%\section{CGO's with Logarithmic phase}

\bibliographystyle{plain}
\bibliography{BiLaplace_Enclosure}

\end{document}